\documentclass[reqno,12pt,a4paper]{amsart}
\usepackage{pgfplots}
\pgfplotsset{/pgf/number format/use comma, compat=newest}
\usepackage{cancel}
\usepackage[english]{babel}
\usepackage{amsmath,amsthm,amssymb,amsfonts}
\usepackage{scrtime}
\usepackage{enumitem}
\usepackage{comment}

\usepackage{color}
\usepackage{ifpdf}
\ifpdf
\usepackage{hyperref}
\else
\usepackage[hypertex]{hyperref}
\fi
\usepackage{xstring}
\usepackage{ifthen}
\usepackage{siunitx}

\usepackage{amsmath,amscd,mathabx}
\usepackage{pdfsync,verbatim}

\usepackage{color}
\usepackage[T1]{fontenc}
\usepackage{graphics}

\usepackage{a4,latexsym,amssymb,amsfonts}
\usepackage{cases}

\usepackage{cancel}
\usepackage{pgfplots}
\pgfplotsset{/pgf/number format/use comma, compat=newest}
\usepackage{cancel}
\usepackage[english]{babel}

\usepackage{hyperref}
\usepackage{xstring}\usepackage{ifthen}\usepackage{siunitx}

\usepackage{amsmath,amscd} 

\usepackage{pdfsync} 
\usepackage{color}

\usepackage{enumerate}
\usepackage{amsthm}

\usepackage{pdfsync,verbatim}

\usepackage{comment}

\numberwithin{equation}{section}

\newtheorem{theorem}{Theorem}[section]

\theoremstyle{definition}

\newcommand\RR{{\mathbb{R}}}

\newcommand\I{{\mathcal{I}}}

\title[Restriction estimates for the free
group]{Restriction estimates for the free two step nilpotent
group on three generators }

\author{Valentina Casarino}
\address{Universit\`a degli Studi di Padova\\Stradella san Nicola
3\\I-36100 Vicenza, Italy}
\email {valentina.casarino@unipd.it}

\author{Paolo Ciatti}
\address{Universit\`a degli Studi di Padova\\Via Marzolo 9\\I-35100
Padova, Italy}
\email{paolo.ciatti@unipd.it}

\thanks{Research supported by the Italian \emph{Ministero dell'Istruzione, dell'Universit\`a e della Ricerca} through PRIN \emph{``Real and Complex Manifolds: Geometry, Topology
and Harmonic Analysis''} and the GNAMPA project
2016 \emph{
``Calcolo funzionale per operatori subellittici su variet\`a''}.  
}

\begin{document}
\begin{abstract}
Let $G$ be the free two step nilpotent Lie group on three
generators and let $L$ be a subLaplacian on it.
We compute the spectral resolution of $L$
and prove that the operators arising from this decomposition
enjoy a Tomas-Stein type estimate.
\end{abstract}
\maketitle

\date{\today, \thistime}
%\maketitle

%\begin{comment}

\section{Introduction}\label{introduction}

%Since when E. Stein observed that, 
Starting from the observation of E. Stein
that, when
the Lebesgue exponent $p$ is sufficiently close to $1$, the Fourier transform of an $L^p$ function restricts in a sense, that may be made precise,
%, in a sense that can be made precise,
to a %sufficiently curved
compact hypersurface of nonvanishing curvature,
various forms of
restriction theorems for the Fourier transform
became one of 
the main theme of %contemporary 
analysis.

%Up to now
%In this theory  hitherto 
The most satisfactory result obtained so far in this theory
is the Stein-Tomas theorem, which concerns the restriction
in the $L^2$- sense
of the Fourier transform of a function in $L^p(\mathbb R^n)$,
with $1\leq p \leq 2\frac{n+1}{n+3}$, to the sphere $S^{n-1}$.
This result was proved in the mid seventies of the last century
and, as any very important theorem, has now a great deal
of extensions and applications in different
branches of mathematics.
In particular, it was observed by Stein himself and by
R. Strichartz in \cite{Str} that the Stein-Tomas estimates can be 
interpreted
as bounds concerning the mapping properties
between Lebesgue spaces
of the operators arising in the
spectral decomposition of the Euclidean Laplacian.

This point of view is emphasised in the works of
C. Sogge, who studied the boundedness of
the spectral projections of the Laplace-Beltrami
operator on the spheres in \cite{Sogge87} and more generally
on compact Riemannian manifolds
in \cite{Sogge88}. 
%(see also \cite{C1}, \cite{C2}, \cite{CC2}).
Few years later D. M\"uller proved for the first time a result of this sort for a subelliptic operator, the subLaplacian on the Heisenberg group.
In the series of papers \cite{C1}, \cite{C2},
\cite{CC}, \cite{CC3} the authors of this article studied a sort of combination
of the previous results considering estimates
for the joint spectral projections
of a Laplace-Beltrami operator and
a subLaplacian.

\begin{comment}
Some years ago in \cite{CC} and \cite{CC2} we extended the result of M\"uller
to the subLaplacians on a larger class of
two step nilpotent Lie groups enjoying
a special nondegeneracy condition.
These are groups in which the quotient of
the Lie algebra with respect to a hyperplane contained
in the center is always isomorphic to a Heisenberg
algebra.
\end{comment}

Some years ago we started in \cite{CC} and \cite{CC2}
an investigation devoted to extend the result of M\"uller to subLaplacians on more general
two step nilpotent Lie groups.
In those papers %however
 we considered %only
  groups
enjoying a special nondegeneracy condition,
according to which the quotient %of the Lie algebra
with respect to a codimension one subspace of the center
is isomorphic to a Heisenberg group. %algebra.

\begin{comment}
we study a subLaplacian associated
to the free two step nilpotent group on three generators.
This example differs from the groups we studied previously
since the group

Some years ago in \cite{CC} and \cite{CC2} we extended the result of M\"uller to the subLaplacians on a two step nilpotent Lie group that enjoys the so called M\'etivier condition.  
This is a nondegeneracy property of the Lie algebra of the group which corresponds to require that its quotient with respect to a hyperplane contained in the center is always isomorphic to a Heisenberg algebra.
\end{comment}

In this paper we are instead concerned with the mapping properties
of the operators arising in
the spectral decomposition of an invariant subLaplacian
on the free two step nilpotent Lie group on three
generators.
In this group the quotient with respect to a codimension one subspace %hyperplane
of the center is isomorphic to the direct product of the
three dimensional Heisenberg group and the real line.
Therefore, the nondegeneracy property
we exploited in the previous works here is absent. 
\begin{comment}

this is an example of a more general
class of groups than those discussed so far.

Therefore, this is a first example of a group
in which the M\'etievier
property does not hold.
\end{comment}
However, the analysis on this group is made easier 
%facilitated
by the fact that the abelian component
in the quotients with respect to the planes in the center is always one dimensional
and by the fact that the combined %simultaneous
 action of the rotation
group on the two layers of the Lie algebra gives rise to a
family of
automorphisms.

\medskip

We conclude the introduction with a short description of the next sections.

In Section 2 we describe some features of the group
$G$
we are concerned with.
Decomposing its Lie algebra $\mathfrak g$ %of $G$
 into the direct sum 
%$\mathfrak g=
 $\mathfrak z\oplus\mathfrak v$, where $\mathfrak z$ is the center,
we see that once a nontrivial linear form
$\mu$ on $\mathfrak z$ has been fixed,
the quotient of $\mathfrak g$ with respect to the null space
of $\mu$ is the direct product of a three dimensional
Heisenberg algebra times %a copy of
 the real line.
Moreover, these two components, the Heisenberg and the abelian one,
depend only on the line in $\mathfrak z^*$ spanned by $\mu$.%\footnote{\color{blue} Quest'ultima frase \`e da omettersi, forse non serve tanto.}

Then in Section 3 we derive the spectral decomposition
of an $L^2$ function $f$ with respect to the subLaplacian. This
decomposition is expressed first of all in terms of the
Fourier transform $f^\mu$ of $f$ on 
$\mathfrak z$,
which for every $\mu$ in the dual of $\mathfrak z$
is a function living on $\mathfrak v$.
% inverse Fourier transform of $f$ written in polar coordinates in the dual of the center of the Lie algebra.  This is
This transform is followed by the section of the Fourier
transform of $f^\mu$ in
the direction of the radical of $\mu$,
which is finally decomposed into eigenfunctions
of the two dimensional twisted Laplacian.

Finally, in Section 4 we prove that the operators arising from the
spectral decomposition of the subLaplacian
are bounded for
 $1\leq s\leq 6/5$ and $1\leq p\leq 2$ from the nonisotropic Lebesgue
space $L^s_\mathfrak z L^p_\mathfrak v$, in which the integrations in $\mathfrak z$
and $\mathfrak v$ are weighted with
exponents $s$ and $p$, to $L^{s'}_\mathfrak z L^2_\mathfrak v$.
The basic ingredients in these estimates are provided
by the Stein-Tomas theorem, which in fact
dictates the range of the exponents concerning
the integration on $\mathfrak z$,
and by the estimates proved some years ago
by H. Koch and F. Ricci in \cite{KR} for the twisted Laplacian.
The range of exponents for which our estimates
hold is sharp as examples analogue to those
provided in \cite{CC} show.

\section{Generalities}\label{generalities}

Let $G$ be a free two step nilpotent Lie group
on three generators.
We assume that $G$ is connected and simply connected.

The Lie algebra, $\mathfrak g$, of $G$ splits as
a vector space into the direct sum
$\mathfrak g =\mathfrak v \oplus \mathfrak z$, where
$\mathfrak z$ is the centre.
Both $\mathfrak v$ and $\mathfrak z$ are three dimensional vector spaces.
To proceed it is convenient to introduce on $\mathfrak g$ an inner product $\langle \cdot,
\cdot \rangle$ with respect to which $\mathfrak v$ and $\mathfrak z$ are orthogonal subspaces.
The inner product induces a norm on $\mathfrak g$ and a norm on  $\mathfrak g^*$, the space of linear forms on $\mathfrak g$, which we will both denote by $|\cdot|$.

We shall always identify $G$ with its Lie algebra $\mathfrak g$
by means of the exponential mapping and
use coordinates $(x, z)$ with $x = (x_1,x_2,x_3)\in
\mathbb R^3$ labeling the points of $\mathfrak v$
and $z=(z_1,z_2,z_3)\in
\mathbb R^3$ the points of $\mathfrak z$.
The vector fields
\begin{equation*}
X_1
=
{\partial }_{x_1} + \frac{x_3}2 {\partial }_{z_2}- \frac{x_2}2 {\partial }_{z_3},\quad
X_2
=
{\partial }_{x_2} + \frac{x_1}2 {\partial }_{z_3}- \frac{x_3}2 {\partial }_{z_1},\quad
X_3
=
{\partial }_{x_3} + \frac{x_2}2 {\partial }_{z_1}- \frac{x_1}2 {\partial }_{z_2}
 \end{equation*}
are left invariant and satisfy
\begin{equation}\label{Liebr}
[X_1,X_2] 
=
\partial_{z_3} =Z_3,
\quad
[X_2,X_3] 
=
\partial_{z_1} = Z_1,
\quad
[X_3,X_1] 
=
\partial_{z_2} = Z_2.
\end{equation}

Fix a point $\omega$ in
the unit sphere
$S = \{\nu \in \mathfrak z^*: |\nu| =1\}$;
here $\mathfrak z^*$ denotes the space of linear forms on $\mathfrak z$.
We call $\mathfrak r_\omega$ the radical of the skew-symmetric %bilinear
 form
$$
\mathfrak v\times\mathfrak v \ni(X, Y) \mapsto \omega([X,Y]),
$$
that is the space
$$
\mathfrak r_\omega = \{X\in \mathfrak v:
\omega ([X,Y]) = 0 \, \text{for all}\, Y\in \mathfrak v\},
$$
which as follows from \eqref{Liebr} is one dimensional. 

If $Z_\omega\in\mathfrak z$ satisfies $\omega(Z_\omega) = 1$, then $|Z_\omega|=1$ and
\begin{equation*}
[X,Y] = \omega([X,Y]) Z_\omega,
\quad
X,Y \in \mathfrak v.
\end{equation*}
Thus if we fix two unit vectors $X_\omega$ and $Y_\omega$ in
$\mathfrak v_\omega$
satisfying $\omega([X_\omega,Y_\omega])=1$,
then the subspace $\mathfrak h_\omega$ spanned
by $\{X_\omega,Y_\omega,Z_\omega\}$
is a Lie algebra isomorphic to the three dimensional
Heisenberg algebra.
 
The subspace $\mathfrak k_\omega$ of $\mathfrak z$
orthogonal to $Z_\omega$ coincides with the kernel of $\omega$.
Therefore the quotient of $\mathfrak g$
with respect to $\mathfrak k_\omega$ is isomorphic
as a Lie algebra to the direct product of
$\mathfrak h_\omega$ and $\mathfrak r_\omega$.
To equip this algebra with a system of coordinates, we decompose a vector
into the sum
$$
v_\omega V_\omega
+x_\omega X_\omega+y_\omega Y_\omega
+z_\omega Z_\omega,
$$
where $V_\omega$ is a unit vector 
in $\mathfrak r_\omega$ and
$(v_\omega,x_\omega,y_\omega,z_\omega) \in \mathbb R^4$.
We shall identify, with a slight abuse of notation, $\omega$ with the linear form in $\mathfrak v^*$ canonically associated
to $V_\omega$, 
%mapping $X \mapsto \langle V_\omega, X\rangle V_\omega$,
 writing $v_\omega=\omega(X)$.
 
Fixing two orthonormal vectors $\{W_{\omega,1}, W_{\omega,2}\}$ in $\mathfrak k_\omega$, we obtain an orthogonal basis $\{Z_\omega, W_{\omega,1}, W_{\omega,2}\}$ of $\mathfrak z$.
Then any vector in $\mathfrak g$ can be uniquely written as
$$
v_\omega V_\omega+x_\omega X_\omega+y_\omega Y_\omega+z_\omega Z_\omega
+ w_{\omega,1}W_{\omega,1}+w_{\omega,2}W_{\omega,2},
$$
where
$(v_\omega,x_\omega,y_\omega,z_\omega,
w_{\omega,1},w_{\omega,2}) \in \mathbb R^6$
are coordinates adapted to $\omega$.

The transformation mapping
$X=(x_1,x_2,x_3)$ to $(x_\omega, y_\omega, v_\omega)$ is a rotation $R_\omega \in SO(3)$ fixing $V_\omega$ 
%We observe that, $R_\omega V_\omega = V_\omega$
%$(x_\omega,y_\omega)=R_\omega(X - \omega(X) V_\omega)=R_\omega(X) - \omega(X) V_\omega$
and hence 
\begin{equation*}
(x_\omega,y_\omega)=x_\omega X_\omega+y_\omega Y_\omega
=R_\omega(X) - \omega(X) V_\omega.
\end{equation*}
Since by the universal property rotations of $\mathfrak v$ extend
to automorphisms of $\mathfrak g$, $R_\omega$
 preserves $Z_\omega$ and
maps $(z_1,z_2,z_3)$ to $(
w_{\omega,1},w_{\omega,2},z_\omega)$.
%that is we have
%\begin{align}\label{Romega}\left(\begin{array}{c}x_\omega \\ y_\omega \\ v_\omega\end{array}\right)= R_\omega \left(\begin{array}{c}x_1\\ x_2\\ x_3\end{array}\right),
%\quad\quad\left(\begin{array}{c}w_{\omega,1}\\ w_{\omega,2}\\ z_{\omega}\end{array}\right)= R_\omega\left(\begin{array}{c}z_1\\ z_2\\ z_3\end{array}\right).\end{align}

\section{The subLaplacian}

The vector fields $X_1$, $X_2$, $X_3$
generate by \eqref{Liebr}  the entire Lie algebra
$\mathfrak g$.
Therefore by H\"ormander's theorem
the subLaplacian
$$
L = -X_1^2-X_2^2-X_3^2
$$
is a hypoelliptic differential operator.
It is easy to see that
\begin{align*}
L &= -\partial_{x_1}^2 -\partial_{x_2}^2 -\partial_{x_3}^2 + \frac 12 (x_2\partial_{x_3}
- x_3\partial_{x_1})\partial_{z_1}+ \frac 12 (x_3\partial_{x_1}
- x_2\partial_{x_3})\partial_{z_2}+ \frac 12 (x_1\partial_{x_2}
- x_2\partial_{x_1})\partial_{z_3}\\
&- \frac 14 (x_2^2+ x_3^2)\partial_{z_1}^2- \frac 14 (x_1^2+ x_3^2)\partial_{z_2}^2
- \frac 14 (x_1^2+ x_2^2)\partial_{z_3}^2
\end{align*}
and hence that $L$ is
homogeneous with respect to 
the dilations defined by
 $\delta_\epsilon (X,Z) = (\epsilon x,
\epsilon^2 z)$, $\epsilon >0$, that is
$$
L(f\circ \delta_\epsilon) = \epsilon^2 (Lf)\circ \delta_\epsilon.
$$

The subLaplacian is a symmetric operator on the Schwartz class
and extends to a self-adjoint operator on $L^2(G)$
with positive spectrum.
Our plan in the section consists in deriving its spectral decomposition.
Our construction is largely inspired by the analogous derivations
in \cite{ACDS} and \cite{MM}.
%the spectral decomposition of $L$.
To accomplish the task, given a Schwartz function $f$ on $G$, we take
its partial Fourier transform in the central variables
$$
f^\mu (X) = \int_{\mathfrak z^*} e^{-i \mu(Z)} f(X,Z) dZ,
\qquad \mu\in \mathfrak z ^*.
$$

We then introduce spherical coordinates
in $\mathfrak z ^*$, writing $\mu = \rho \omega$,
where $\rho = |\mu|$ and $\omega$
belongs to $S$.
In these coordinates the Fourier inversion formula in $\mathfrak z^*$ reads
\begin{align}\label{fourierinv}
f(X,Z) = \int_S \int_{0}^\infty e^{i \rho \omega(Z)} f^{\rho \omega} (X) \rho^{2} d\rho
d\omega.
\end{align}
In particular, we have
\begin{align}\label{Xf}
X_a f(X,Z) = \int_S \int_{0}^\infty e^{i \rho \omega(Z)} X_a^{\rho\omega}f^{\rho \omega} (X) \rho^{2} d\rho
d\omega,
\qquad a=1,2,3,
\end{align}
where the vector fields $X_a^{\rho\omega}$ are the differential operators on $\mathfrak v$ 
defined by
\begin{equation*}%\label{L e Deltat}
X_a\left(e^{-i \rho\omega(Z)} g(X)\right) = e^{-i \rho\omega(Z)} X_a^{\rho\omega} g(X),\quad
a =1,2,3,
\end{equation*}
for any Schwartz function $g$. %on $\mathfrak v$.
They are more explicitly given by
\begin{equation*}
X^{\rho\omega}_1
=
{\partial }_{x_1} + \frac{i}2\rho (\omega_3x_2-\omega_2x_3),\quad
X^{\rho\omega}_2
=
{\partial }_{x_2} + \frac{i}2\rho (\omega_1x_3-\omega_3x_1),\quad
X^{\rho\omega}_3
=
{\partial }_{x_3} + \frac{i}2\rho (\omega_2x_1-\omega_1x_2),
 \end{equation*}
where $\omega_a = \omega\left( Z_a\right)$,
$a=1,2,3$.
Defininig the differential operator
 $L^{\rho\omega}$ on $\mathfrak v$ by
\begin{equation}\label{L e Deltat}
L\left(e^{-i \rho\omega(Z)} g(X)\right) = e^{-i \rho\omega(Z)} L^{\rho\omega} g(X),
\end{equation}
we find
\begin{align*}
L^{\rho\omega} &= -\partial_{x_1}^2 -\partial_{x_2}^2 -\partial_{x_3}^2 + \frac i2\rho \omega_1 (x_3\partial_{x_1}-x_2\partial_{x_3}
)+ \frac i2\rho \omega_2 (x_2\partial_{x_3}-x_3\partial_{x_1}
)+ \frac i2\rho \omega_3 (x_2\partial_{x_1}-x_1\partial_{x_2}
)\\
&+ \frac 14 \rho^2\omega_1^2 (x_2^2+ x_3^2)+\frac 14 \rho^2\omega_2^2 (x_1^2+ x_3^2)
+ \frac 14 \rho^2\omega_3^2 \rho^2 (x_1^2+ x_2^2).
\end{align*}
In the coordinates $(v_\omega, x_\omega, y_\omega, z_\omega,
w_{\omega,1}, w_{\omega,2})$ 
%adapted to $\omega$ 
the operator $L^{\rho\omega}$
becomes
\begin{align}\label{Lomegainadattate}
L^{\rho\omega}_{x_\omega,y_\omega,v_\omega} 
&= %L^{\rho\omega}=
 -\partial_{v_\omega}^2 -\partial_{x_\omega}^2 -\partial_{y_\omega}^2 
+ \frac i2  \rho (y_\omega\partial_{x_\omega}-x_\omega\partial_{y_\omega}
)+ \frac 14 \rho^2 (x_\omega^2+ y_\omega^2)
\\\notag
&= -\partial_{v_\omega}^2 +\Delta^{\rho}_{x_\omega,y_\omega},
\end{align}
where
\begin{align*}
\Delta_{s, t}^{\lambda} =
-\frac{\partial^2}{\partial s^2}-\frac{\partial^2}{\partial t^2}
+ \frac i2 \lambda
\left(t \frac{\partial}{\partial s} -s \frac{\partial}{\partial t} \right)
+ \frac {\lambda^2}4 \left( s^2 + t^2\right)
\end{align*}
is the two dimensional $\lambda$-twisted Laplacian.
%$\Delta^{\rho}_{x_\omega,y_\omega}$ is the two dimensional $\rho$-twisted Laplacian.

The spectrum of $\Delta_{s, t}^{\lambda}$
consists of eigenvalues,
which are given by $\lambda(2k+1)$ with $k =0,1,2,\cdots$.
We write the spectral decomposition of $g$ 
 %into eigenfunctions of
 with respect to $\Delta_{s, t}^{\lambda}$ as
\begin{equation}\label{decomposizione del lapltw}
g(s,t) = \sum\limits_{k=0}^\infty \Lambda_k^\lambda g(s,t),
\end{equation}
where $\Lambda_k^\lambda$ %is the operator
maps $L^2(\mathbb R^2)$ onto the eigenspace
associated to $\lambda(2k+1)$.
These are integral operators (for more details about them see, for instance, \cite[p.19 ff]{Th})
satisfying the estimates
\begin{equation}\label{KR 1}
%\nu_p=
\|\Lambda^\lambda_{k} f\|_{L^{2}(\RR^{2})}
%\rightarrow L^{2}(\RR^{2})} \approx
\leq C \lambda^{\frac1p -\frac12}
(2k+1)^{\gamma(\frac1p)} \| f\|_{L^{p}(\RR^{2})}\,,
\quad
 1\le p\le 2\,,
\end{equation}
proved by H. Koch and F. Ricci in \cite{KR};
here $\gamma$
is the piecewise
affine function on $[\frac12,1]$
defined by
\begin{equation}\label{gamma}
%\label{KR stime}
\gamma\left (\frac1p \right):=
\begin{cases}\!
%{\left(\frac1p-\frac{1}{2}\right)-\frac{1}{2} }
{\frac1p-1}
&\,\text{if
  $1\le p\leq \frac65$,}\cr
{\frac{1}{2}(\frac{1}{2}-\frac1p)}
&\,\text{if
  $ \frac65\le p\le 2$.}\cr
\end{cases}
\end{equation}

In order to decompose
$f^{\rho \omega}$ into eigenfunctions of
the twisted Laplacian,
we define the function $g_{\rho,\omega} = f^{\rho \omega}\circ R_\omega^{-1}$, where $R_\omega^{-1}$ is the inverse of $R_\omega$.
Since $(x_\omega, y_\omega, v_\omega) = R_\omega
(x_1,x_2,x_3)= R_\omega X$, recalling that $v_\omega= \omega(X)$
and $(x_\omega,y_\omega)=R_\omega(X) - \omega(X) V_\omega$, we have $f^{\rho \omega}(X) =g_{\rho,\omega}(x_\omega, y_\omega, v_\omega)=g_{\rho,\omega}(R_\omega(X) - \omega(X) V_\omega, \omega(X))$.
Then it makes sense to decompose the function $g_{\rho,\omega}(x_\omega, y_\omega, v_\omega)$ in eigenfuctions of $\Delta_{x_\omega, y_\omega}^{\rho}$, obtaining
\begin{align*}
f^{\rho \omega}(X) &=g_{\rho,\omega}(x_\omega, y_\omega, v_\omega)
\\
&=\sum\limits_{k=0}^\infty \Lambda_k^\rho g_{\rho,\omega}(x_\omega, y_\omega, v_\omega)
\\
&=\sum\limits_{k=0}^\infty \Lambda_k^\rho 
(f^{\rho \omega}\circ R_\omega^{-1})
\big(R_\omega(X) - \omega(X) V_\omega, \omega(X)\big).
\end{align*}

Taking the partial Fourier transform of
$g_{\rho,\omega}$ in $\mathfrak r_\omega$, 
it follows that
\begin{align*}
\mathfrak F_\omega g_{\rho,\omega}(x_\omega, y_\omega;\xi) %=\mathfrak F_\omega g_{\rho,\omega}(R_\omega(X) - \omega(X) V_\omega;\xi)
&= \int_{-\infty}^\infty e^{-i s\xi} g_{\rho,\omega} 
\big(R_\omega(X) - \omega(X) V_\omega+sV_\omega\big) ds
\\
&=
\sum\limits_{k=0}^\infty 
\int_{-\infty}^\infty e^{-i s\xi} \Lambda_k^\rho g_{\rho,\omega}
\big(R_\omega(X) - \omega(X) V_\omega+sV_\omega\big) ds,
\end{align*}
and also
\begin{align*}
\mathfrak F_\omega g_{\rho,\omega}(x_\omega, y_\omega;\xi) %=\mathfrak F_\omega g_{\rho,\omega}(R_\omega(X) - \omega(X) V_\omega;\xi)
&=
\sum\limits_{k=0}^\infty \Lambda_k^\rho \mathfrak F_\omega
g_{\rho,\omega}
\big(R_\omega(X) - \omega(X) V_\omega;\xi\big),
\end{align*}
since the operators $\Lambda_k^\rho$ and $\mathfrak F_\omega$, acting on different spaces, namely
$\mathfrak v_\omega$ and $\mathfrak r_\omega$,
 commute.
Thus, inverting the Fourier transform we obtain
%\footnote{\color{blue} Una delle due righe va cancellata.}
\begin{align*}%\label{fourierinrad}
g_{\rho,\omega}(x_\omega, y_\omega,v_\omega)
%&=\sum\limits_{k=0}^\infty \int_{-\infty}^\infty e^{i v_\omega \xi}\mathfrak F_\omega \Lambda_k^\rho g_{\rho,\omega}(x_\omega, y_\omega;\xi)  d\xi\\
&=
\sum\limits_{k=0}^\infty 
\int_{-\infty}^\infty e^{i v_\omega \xi}
\Lambda_k^\rho \mathfrak F_\omega g_{\rho,\omega}(x_\omega, y_\omega;\xi)  d\xi.
\end{align*}
Hence, we have
%\footnote{\color{blue} Una delle due righe va cancellata.}
\begin{align}\label{fourierinrad}
f^{\rho \omega}(X)  %\notag
%&=\sum\limits_{k=0}^\infty \int_{-\infty}^\infty e^{i \omega(X) \xi}\mathfrak F_\omega \Lambda_k^\rho (f^{\rho \omega}\circ R_\omega^{-1})\big(R_\omega(X) - \omega(X) V_\omega; \xi\big)  d\xi\\
&=
\sum\limits_{k=0}^\infty 
\int_{-\infty}^\infty e^{i \omega(X) \xi}
\Lambda_k^\rho \mathfrak F_\omega 
(f^{\rho \omega}\circ R_\omega^{-1})
\big(R_\omega(X) - \omega(X) V_\omega; \xi\big)  d\xi,
\end{align}
which plugged in \eqref{fourierinv} yields
\begin{align}\label{f}
f(X,Z) = \sum\limits_{k=0}^\infty \int_S \int_{0}^\infty e^{i \rho \omega(Z)} \int_{-\infty}^\infty e^{i \omega(X) \xi}
\Lambda_k^\rho \mathfrak F_\omega 
(f^{\rho \omega}\circ R_\omega^{-1})
\big(R_\omega(X) - \omega(X) V_\omega; \xi\big)  d\xi \rho^{2} d\rho
d\omega.
\end{align}

We may now deduce from \eqref{f} the spectral
decomposition of $f$ with respect to $L$ by
replacing the arguments of the sum and the integrals
with generalised eigenfunctions.
To do that we notice that by \eqref{Xf} we have
\begin{align*}%\label{Lf}
L f(X,Z) = \int_S \int_{0}^\infty e^{i \rho \omega(Z)} L^{\rho\omega}f^{\rho \omega} (X) \rho^{2} d\rho
d\omega,
\end{align*}
which by \eqref{fourierinrad} and \eqref{Lomegainadattate}
implies  %coincides with
\begin{align*}
&\int_S \int_{0}^\infty e^{i \rho \omega(Z)} L^{\rho\omega}
\left(
\sum\limits_{k=0}^\infty 
\int_{-\infty}^\infty e^{i \omega(X) \xi}
\Lambda_k^\rho \mathfrak F_\omega 
(f^{\rho \omega}\circ R_\omega^{-1})
\big(R_\omega(X) - \omega(X) V_\omega; \xi\big)  d\xi
\right) \rho^{2} d\rho
d\omega
%\\&=\sum\limits_{k=0}^\infty \int_S \int_{0}^\infty e^{i \rho \omega(Z)}\left( L^{\rho\omega}_{x_\omega,y_\omega,v_\omega}\int_{-\infty}^\infty e^{i v_\omega \xi}\Lambda_k^\rho \mathfrak F_\omega (f^{\rho \omega}\circ R_\omega^{-1})
%%\big( x_\omega, y_\omega; \xi\big)  d\xi
%\big( \cdot,\cdot; \xi\big)  d\xi\right)(R_\omega(X)) \rho^{2} d\rhod\omega
\\&=
\sum\limits_{k=0}^\infty 
\int_S \int_{0}^\infty e^{i \rho \omega(Z)}
\\&\times
\left( (-\partial_{v_\omega}^2 +\Delta^{\rho}_{x_\omega,y_\omega})
\int_{-\infty}^\infty e^{i v_\omega \xi}
\Lambda_k^\rho \mathfrak F_\omega 
(f^{\rho \omega}\circ R_\omega^{-1})
\big(x_\omega,y_\omega; \xi\big)  d\xi
%\big( x_\omega, y_\omega; \xi\big)  d\xi
\right)\Big|_{R_\omega(X)}
 \rho^{2} d\rho
d\omega
\\&=
\sum\limits_{k=0}^\infty 
\int_S \int_{0}^\infty e^{i \rho \omega(Z)}
\\&
\times
\left(
\int_{-\infty}^\infty e^{i v_\omega \xi}
(\xi^2+\Delta^{\rho}_{x_\omega,y_\omega})
\Lambda_k^\rho \mathfrak F_\omega 
(f^{\rho \omega}\circ R_\omega^{-1})
\big( x_\omega,y_\omega; \xi\big)  d\xi
%\big( x_\omega, y_\omega; \xi\big)  d\xi
\right)\Big|_{R_\omega(X)}
 \rho^{2} d\rho
d\omega
\\&=
\sum\limits_{k=0}^\infty 
\int_S \int_{0}^\infty e^{i \rho \omega(Z)}
\\&
\times
\left(
\int_{-\infty}^\infty e^{i v_\omega \xi}
(\xi^2+\rho(2k+1))\Lambda_k^\rho \mathfrak F_\omega 
(f^{\rho \omega}\circ R_\omega^{-1})
%\big( \cdot,\cdot; \xi\big)  d\xi
\big( x_\omega, y_\omega; \xi\big)  d\xi
\right)
\Big|_{R_\omega(X)}
\rho^{2} d\rho
d\omega
\\&=
\sum\limits_{k=0}^\infty 
\int_S \int_{0}^\infty e^{i \rho \omega(Z)}
\\&
\times
\left(
\int_{-\infty}^\infty e^{i \omega(X) \xi}
(\xi^2+\rho(2k+1))\Lambda_k^\rho \mathfrak F_\omega 
(f^{\rho \omega}\circ R_\omega^{-1})
\big(R_\omega(X)-\omega(X)V_\omega; \xi\big)  d\xi
\right) \rho^{2} d\rho
d\omega,
\end{align*}
where $\big|_{R_\omega(X)}$ means that
the integral must be computed in
$x_\omega X_\omega+ y_\omega Y_\omega= R_\omega(X)- \omega(X) V_\omega$ and $v_\omega=\omega(X)$.

Replacing in the last integral $\rho(2k+1)$
with $\rho$ we obtain
\begin{align*}
Lf(X,Z) &=
\int_S \sum\limits_{k=0}^\infty (2k+1)^{-3}
\int_{-\infty}^\infty \int_{0}^\infty  e^{i \rho \omega(Z)/(2k+1)}
e^{i \omega(X) \xi}(\xi^2+\rho)
\\&
\times
\Lambda_k^{\rho/(2k+1)} \mathfrak F_\omega 
(f^{\rho \omega/(2k+1)}\circ R_\omega^{-1})
\big(R_\omega(X)-\omega(X)V_\omega; \xi\big)
 \rho^{2} d\rho d\xi
d\omega,
\end{align*}
from which, setting $\mu=\xi^2+\rho$
and using Fubini's theorem, we deduce
\begin{align*}
Lf(X,Z) &=
\int_S \sum\limits_{k=0}^\infty (2k+1)^{-3}
\int_{-\infty}^\infty \int_{\xi^2}^\infty  e^{i (\mu-\xi^2) \omega(Z)/(2k+1)}
e^{i \omega(X) \xi}(\mu-\xi^2)^{2}
\\&
\times
\mu \Lambda_k^{(\mu-\xi^2)/(2k+1)} \mathfrak F_\omega 
(f^{(\mu-\xi^2) \omega/(2k+1)}\circ R_\omega^{-1})
\big(R_\omega(X)-\omega(X)V_\omega; \xi\big)
d\mu d\xi
d\omega
\\&= \int_{0}^\infty \mu
\int_S \sum\limits_{k=0}^\infty (2k+1)^{-3}
\int_{-\sqrt\mu}^{\sqrt\mu} e^{i (\mu-\xi^2) \omega(Z)/(2k+1)}
e^{i \omega(X) \xi} (\mu-\xi^2)^{2}
\\&
\times
\Lambda_k^{(\mu-\xi^2)/(2k+1)} \mathfrak F_\omega 
(f^{(\mu-\xi^2) \omega/(2k+1)}\circ R_\omega^{-1})
\big(R_\omega(X)-\omega(X)V_\omega; \xi\big)
d\xi
d\omega d\mu.
\end{align*}

The expression for $Lf$ just derived shows that
if we write
\begin{equation}\label{dec}
f(X,Z) = \int\limits_0^\infty \mathcal P_\mu f(X,Z) d\mu,
\end{equation}
where
\begin{align}\label{P}
\mathcal P_\mu f(X,Z) &=
\int_S \sum\limits_{k=0}^\infty (2k+1)^{-3}
\int_{-\sqrt\mu}^{\sqrt\mu} e^{i (\mu-\xi^2) \omega(Z)/(2k+1)}
e^{i \omega(X) \xi} (\mu-\xi^2)^{2}
\\\notag& 
\times
\Lambda_k^{(\mu-\xi^2)/(2k+1)} \mathfrak F_\omega 
(f^{(\mu-\xi^2) \omega/(2k+1)}\circ R_\omega^{-1})
\big(R_\omega(X)-\omega(X)V_\omega; \xi\big)
 d\xi
d\omega,
\end{align}
then
$$
L \mathcal P_\mu f = \mu \mathcal P_\mu f.
$$
Thus, \eqref{dec} provides
the spectral resolution of $f$
with respect to $L$.
%decomposition we were looking for.

\section{The restriction theorem
%Estimates for $\mathcal P_\mu$
}

In this section we show that the operators
$\mathcal P_\mu$ satisfy some restriction estimates.
To state and prove this result we introduce nonisotropic norms
on $G$ defined, for $1\leq p,s < \infty$, by
$$
\|f\|_{L^s_\mathfrak zL^p_\mathfrak v}
=
\left(\int_{\mathfrak v}
\left(\int_{\mathfrak z} |f(X,Z)|^s dZ
\right)^{\frac ps}dV
\right)^{\frac1p},
$$
with the obvious modifications when $s$ or $p$ equal
$\infty$.
We shall prove the following theorem.

\begin{theorem}\label{teorema} Let $f$ be a Schwartz function on $G$
and let $1\leq p\leq2$ and $1\leq s\leq \frac65$, then
\begin{align}\label{stima}
\|\mathcal P_\mu f\|_{L^{s'}_\mathfrak zL^2_\mathfrak v}
\leq C \mu^{3\left(\frac1s-\frac1{s'}\right)
+3\left(\frac1p-\frac1{2}\right)} \|f\|_{L^s_\mathfrak zL^p_\mathfrak v}.
\end{align}
These estimates are false for $s>\frac65$.
\end{theorem}

\begin{proof} The sharpness of the range of $s$
where the estimates hold may be proved with
a suitable
and easy modification of the example provided in \cite{CC}.
So we prove only the bound \eqref{stima}.

The dependence on $\mu$ in the right hand side of \eqref{stima}
is dictated by the homogeneity of $L$.
Therefore, it suffices to discuss the case $\mu =1$.

To reduce the complexity of the notation
and make the formulas more readable we consider
a tensor function $f(X,Z) = \alpha(Z)\beta(X)$,
with $\alpha$ and $\beta$ Schwartz functions
on $\mathfrak z$ and $\mathfrak v$.
Then \eqref{P} reduces to
\begin{align*}
\mathcal P_1 f(X,Z) &=
\int_S \int_{-1}^{1} \sum\limits_{k=0}^\infty
\frac{(1-\xi^2)^{2}}{(2k+1)^{3}}
e^{i \frac{(1-\xi^2) \omega(Z)}{2k+1}}
e^{i \omega(X) \xi} 
\hat \alpha \left(\frac{(1-\xi^2)\omega}{2k+1}\right)
\\&\qquad\qquad
\times
\Lambda_k^{\frac{1-\xi^2}{2k+1}} \mathfrak F_\omega 
(\beta\circ R_\omega^{-1})
\big(R_\omega(X)-\omega(X)V_\omega; \xi\big)
 d\xi
d\omega,
\end{align*}
where $\hat \alpha$ denotes the Fourier transform
of $\alpha$.

Consider another tensor function
$g(X,Z)=\gamma(Z)\delta(X)$,
with $\|\gamma\|_{L^s_\mathfrak z} =1$ and
$\|\delta\|_{L^q_{\mathfrak v}} =1$, and compute
\begin{align*}
\langle \mathcal P_1 f, g \rangle &=
\int_{\mathfrak v} \int_{\mathfrak z}
P_1 f(X,Z) \overline{g(X,Z)} dZ dX.
\end{align*}
Then we have
\begin{align*}
\langle \mathcal P_1 f, g \rangle &=
\int_{\mathfrak v}
\int_{\mathfrak z}
\Bigg(
\int_S \int_{-1}^{1} \sum\limits_{k=0}^\infty
\frac{(1-\xi^2)^{2}}{(2k+1)^{3}}
e^{i \frac{(1-\xi^2) \omega(Z)}{2k+1}}
e^{i \omega(X) \xi} 
\hat \alpha \left(\frac{(1-\xi^2)\omega}{2k+1}\right)
\\&\qquad\qquad\times
\Lambda_k^{\frac{1-\xi^2}{2k+1}} \mathfrak F_\omega 
(\beta\circ R_\omega^{-1})
\big(R_\omega(X)-\omega(X)V_\omega; \xi\big)
 d\xi
d\omega
\Bigg) \overline{\gamma(Z)}
\overline{\delta(X)} dZdX
\\&=
\int_S \int_{-1}^{1} \sum\limits_{k=0}^\infty
\frac{(1-\xi^2)^{2}}{(2k+1)^{3}}
\hat \alpha \left(\frac{(1-\xi^2)\omega}{2k+1}\right)
\overline{\hat\gamma\left(\frac{(1-\xi^2) \omega}{2k+1}\right)}
\\&\qquad\qquad\times
\Bigg(
\int_{\mathfrak v}
e^{i \omega(X) \xi} 
\Lambda_k^{\frac{1-\xi^2}{2k+1}} \mathfrak F_\omega 
(\beta\circ R_\omega^{-1})
\big(R_\omega(X)-\omega(X)V_\omega; \xi\big)
\overline{\delta(X)} dX 
\Bigg) 
d\xi
d\omega.
\end{align*}

Changing the variables in the integral over $\mathfrak v$
we obtain (since $|\det R_\omega| =1$)
\begin{align*}
\langle \mathcal P_1 f, g \rangle
&=
\int_S \int_{-1}^{1} \sum\limits_{k=0}^\infty
\frac{(1-\xi^2)^{2}}{(2k+1)^{3}}
\hat \alpha \left(\frac{(1-\xi^2)\omega}{2k+1}\right)
\overline{\hat\gamma\left(\frac{(1-\xi^2) \omega}{2k+1}\right)}
\\&\qquad\quad\times
\Bigg(
\int_{\mathfrak v}
e^{i v_\omega \xi} 
\Lambda_k^{\frac{1-\xi^2}{2k+1}} \mathfrak F_\omega 
(\beta\circ R_\omega^{-1})
\big(x_\omega,y_\omega; \xi\big)
\overline{(\delta\circ R_\omega^{-1})(x_\omega,y_\omega,v_\omega)} dx_\omega
dy_\omega dv_\omega
\Bigg) 
d\xi
d\omega
\\
&=
\int_S \int_{-1}^{1} \sum\limits_{k=0}^\infty
\frac{(1-\xi^2)^{2}}{(2k+1)^{3}}
\hat \alpha \left(\frac{(1-\xi^2)\omega}{2k+1}\right)
\overline{\hat\gamma\left(\frac{(1-\xi^2) \omega}{2k+1}\right)}
\\&\qquad\quad\times
\Bigg(
\int_{\mathfrak v_\omega}
\Lambda_k^{\frac{1-\xi^2}{2k+1}} \mathfrak F_\omega 
(\beta\circ R_\omega^{-1})
\big(x_\omega,y_\omega; \xi\big)
\overline{\mathfrak F_\omega(\delta\circ R_\omega^{-1})(x_\omega,y_\omega;\xi)} dx_\omega
dy_\omega
\Bigg) 
d\xi
d\omega.
\end{align*}

Setting
\begin{align}\label{Psi}
\Psi(\omega;\xi,k)=
\int_{\mathfrak v_\omega}
\Lambda_k^{\frac{1-\xi^2}{2k+1}} \mathfrak F_\omega 
(\beta\circ R_\omega^{-1})
\big(x_\omega,y_\omega; \xi\big)
\overline{\mathfrak F_\omega(\delta\circ R_\omega^{-1})(x_\omega,y_\omega;\xi)} dx_\omega
dy_\omega,
\end{align}
we write
\begin{align*}
\langle \mathcal P_1 f, g \rangle
&=
\sum\limits_{k=0}^\infty
\frac{(1-\xi^2)^{2}}{(2k+1)^{3}}
\int_{-1}^{1} 
\left(\int_S
\hat \alpha \left(\frac{(1-\xi^2)\omega}{2k+1}\right)
\overline{\hat\gamma\left(\frac{(1-\xi^2) \omega}{2k+1}\right)}
\Psi(\omega;\xi,k) d\omega\right)
 d\xi.
\end{align*}
Then the Cauchy-Schwarz inequality yields 
%\footnote{\color{blue} Si pu\`o omettere il passaggio intermedio.}
\begin{align*}
|\langle \mathcal P_1 f, g \rangle|
%&\leq\sum\limits_{k=0}^\infty\int_{-\sqrt\mu}^{\sqrt\mu}\frac{(\mu-\xi^2)^{2}}{(2k+1)^{3}}\left(\int_S\left|\hat \alpha \left(\frac{(\mu-\xi^2)\omega}{2k+1}\right)\right|^2
%|\Psi(\omega;\mu,\xi,k)| 
%d\omega\right)^{\frac12}\\&\qquad\qquad\left(\int_S\left|{\hat\gamma\left(\frac{(\mu-\xi^2) \omega}{2k+1}\right)}\right|^2|\Psi(\omega;\mu,\xi,k)|^2 d\omega\right)^{\frac12} d\xi\\
%%%&\leq\sum\limits_{k=0}^\infty\frac{1}{(2k+1)^{3}}\int_{-1}^{1}(1-\xi^2)^{2}\left(\sup\limits_{\omega\in S}|\Psi(\omega;\xi,k)|\right)\\&\qquad\qquad\left(\int_S\left|\hat \alpha \left(\frac{(1-\xi^2)\omega}{2k+1}\right)\right|^2 d\omega\right)^{\frac12}\left(\int_S\left|{\hat\gamma\left(\frac{(1-\xi^2) \omega}{2k+1}\right)}\right|^2 d\omega\right)^{\frac12} d\xi\\
&\leq
\sum\limits_{k=0}^\infty
\frac{1}{(2k+1)^{3}}
\int_{-1}^{1}
(1-\xi^2)^{2}\,
\Phi(\xi,k)\\
&\qquad\qquad\times
\left(\int_S
\left|\hat \alpha \left(\frac{(1-\xi^2)\omega}{2k+1}\right)
\right|^2 d\omega
\right)^{\frac12}
\left(\int_S
\left|{\hat\gamma\left(\frac{(1-\xi^2) \omega}{2k+1}\right)}
\right|^2 d\omega
\right)^{\frac12} d\xi,
\end{align*}
where we set 
\begin{align}\label{Phi}
\Phi(\xi,k) = \sup\limits_{\omega\in S}
|\Psi(\omega;\xi,k)|
\end{align}
to simplify the notation.
We remind that according to the Tomas-Stein theorem in $\mathbb R^3$ the inequality
\begin{equation*}
\left(\int_S
\left|\hat \eta \left(r \omega\right)
\right|^2 d\omega
\right)^{\frac12} \leq C_s r^{-3/s'} \|\eta\|_{L^s_\mathfrak z}
\end{equation*}
holds for $1\leq s \leq \frac43$ and $r>0$, where $\frac 1s+\frac 1{s'} =1$.
Thus, we obtain (being $\|\gamma\|_{L^s_\mathfrak z} =1$)
\begin{align}\label{5}
|\langle \mathcal P_1 f, g \rangle|
%&\leq \|\alpha\|_{L^s_\mathfrak z} \|\gamma\|_{L^s_\mathfrak z} \sum\limits_{k=0}^\infty \int_{-1}^{1} \frac{(1-\xi^2)^{2}}{(2k+1)^{3}} \Phi(\xi,k)\left(\frac{1-\xi^2}{2k+1}\right)^{-\frac6{s'}}  d\xi \\\notag
&\leq
\|\alpha\|_{L^s_\mathfrak z}
\sum\limits_{k=0}^\infty
\frac{1}{(2k+1)^{3-\frac6{s'}}}
\int_{-1}^{1}
(1-\xi^2)^{2-\frac6{s'}}
\Phi(\xi,k) d\xi.
\end{align}
%\footnote{\color{blue} Questo passaggio \`e da omettere.}

To estimate the integral in the last formula,
we remind the definition \eqref{Phi} of $\Phi$
(see also \eqref{Psi})
and use the Cauchy-Schwarz inequality
in the integral on $\mathfrak v_\omega$
to obtain
\begin{align*}
&\int_{-1}^{1}
(1-\xi^2)^{2-\frac6{s'}}
\Phi(\xi,k) d\xi
\leq
\int_{-1}^{1}
(1-\xi^2)^{2-\frac6{s'}}
\\
&\times
\left|
\int_{\mathfrak v_\omega}
\Lambda_k^{\frac{1-\xi^2}{2k+1}} \mathfrak F_\omega 
(\beta\circ R_\omega^{-1})
\big(x_\omega,y_\omega; \xi\big)
\overline{\mathfrak F_\omega(\delta\circ R_\omega^{-1})(x_\omega,y_\omega;\xi)} dx_\omega
dy_\omega
\right|
d\xi
\\
&\leq
\int_{-1}^{1}
(1-\xi^2)^{2-\frac6{s'}}
\left(
\int_{\mathfrak v_\omega}
\left|\Lambda_k^{\frac{1-\xi^2}{2k+1}} \mathfrak F_\omega 
(\beta\circ R_\omega^{-1})
\big(x_\omega,y_\omega; \xi\big)\right|^2 dx_\omega
dy_\omega
\right)^{\frac12}
\\
&\qquad\qquad\times
\left(
\int_{\mathfrak v_\omega}
\left|\mathfrak F_\omega(\delta\circ R_\omega^{-1})(x_\omega,y_\omega;\xi)\right|^2 dx_\omega
dy_\omega
\right)^{\frac12}
d\xi.
\end{align*}
A further application of the Cauchy-Schwarz inequality
then implies
\begin{align}\label{1}
&\int_{-1}^{1}
(1-\xi^2)^{2-\frac6{s'}}
\Phi(\xi,k) d\xi
\\\notag
&\leq
\left(
\int_{-1}^{1}
(1-\xi^2)^{4-\frac{12}{s'}}
\int_{\mathfrak v_\omega}
\left|\Lambda_k^{\frac{1-\xi^2}{2k+1}} \mathfrak F_\omega 
(\beta\circ R_\omega^{-1})
\big(x_\omega,y_\omega; \xi\big)\right|^2 dx_\omega
dy_\omega d\xi
\right)^{\frac12}
\\\notag
&\qquad\qquad\times
\left(\int_{-\infty}^\infty
\int_{\mathfrak v_\omega}
\left|\mathfrak F_\omega(\delta\circ R_\omega^{-1})(x_\omega,y_\omega;\xi)\right|^2 dx_\omega
dy_\omega
d\xi
\right)^{\frac12}.
\end{align}

Since the Plancherel theorem
applied to $\mathfrak F_\omega$ yields
\begin{align*}
\int_{\mathfrak v_\omega}
\int_{-\infty}^\infty
\left|\mathfrak F_\omega(\delta\circ R_\omega^{-1})(x_\omega,y_\omega;\xi)\right|^2 
d\xi
dx_\omega
dy_\omega
&=
\int_{-\infty}^\infty
\int_{\mathfrak v_\omega}
\left|(\delta\circ R_\omega^{-1})(x_\omega,y_\omega,v_\omega)\right|^2 dx_\omega
dy_\omega
dv_\omega
\\
&=
\|\delta\|_{L^2_{\mathfrak v}}=1,
\end{align*}
(being $|\det R_\omega|=1$),
\eqref{1} reduces to
\begin{align}\label{2}
&\int_{-1}^{1}
(1-\xi^2)^{2-\frac6{s'}}
\Phi(\xi,k) d\xi
\\\notag
&\leq
\left(
\int_{-1}^{1}
(1-\xi^2)^{4-\frac{12}{s'}}
\int_{\mathfrak v_\omega}
\left|\Lambda_k^{\frac{1-\xi^2}{2k+1}} \mathfrak F_\omega 
(\beta\circ R_\omega^{-1})
\big(x_\omega,y_\omega; \xi\big)\right|^2 dx_\omega
dy_\omega d\xi
\right)^{\frac12}.
\end{align}

%In the next step we exploit %in \eqref{2}
Then in force of the Koch-Ricci estimates
\eqref{KR 1} we deduce from \eqref{2} that
\begin{align}\label{3}
&\int_{-1}^{1}
(1-\xi^2)^{2-\frac6{s'}}
\Phi(\xi,k) d\xi
%\\\notag&\leq \sup\limits_{\omega\in S}\left(\int_{-\sqrt\mu}^{\sqrt\mu}(\mu-\xi^2)^{4-\frac{12}{s'}}\left(\frac{\mu-\xi^2}{2k+1}\right)^{2\left(\frac1p-\frac12\right)}\int_{\mathfrak v_\omega}\left|\Lambda_k \mathfrak F_\omega (\beta\circ R_\omega^{-1})\big(x_\omega,y_\omega; \xi\big)\right|^2 dx_\omegady_\omega d\xi\right)^{\frac12}
%\\\notag&
\leq C
\left({2k+1}\right)^{\gamma\left(\frac1p\right)-\left(\frac1p-\frac12\right)}
\\\notag&\times
\left(
\int_{-1}^{1}
(1-\xi^2)^{4-\frac{12}{s'}+2\left(\frac1p-\frac12\right)}
\left(
\int_{\mathfrak v_\omega}
\left|\mathfrak F_\omega 
(\beta\circ R_\omega^{-1})
\big(x_\omega,y_\omega; \xi\big)\right|^p dx_\omega
dy_\omega
\right)^{\frac2p}
 d\xi
\right)^{\frac12}.
\end{align}

Applying the H\"older inequality in \eqref{3} to the integration
 in $\xi$
with exponents $p'/2$ and $p'/(p'-2)= p/(2-p)$,
we deduce that
\begin{align*}%\label{1}
&\int_{-1}^{1}
(1-\xi^2)^{2-\frac6{s'}}
\Phi(\xi,k) d\xi
\\
&\leq
C
\left({2k+1}\right)^{\gamma\left(\frac1p\right)-
\left(\frac1p-\frac12\right)}
\left(
\int_{-1}^{1}
(1-\xi^2)^{\left(1-\frac{3}{s'}\right)\frac {4p}{2-p}+1}
d\xi
\right)^{\frac1p-\frac12}
\\
&
\times\left(
\int_{-\infty}^\infty
\left(
\int_{\mathfrak v_\omega}
\left|\mathfrak F_\omega 
(\beta\circ R_\omega^{-1})
\big(x_\omega,y_\omega; \xi\big)\right|^p dx_\omega
dy_\omega
\right)^{\frac{p'}p}
d\xi
\right)^{\frac1{p'}}.
\end{align*}
The first integral is finite
since $s'\geq 4$, which implies that
\begin{align*}
%\left(4-\frac{12}{s'}\right)\frac p{2-p}&=
\left(1-\frac{3}{s'}\right)\frac {4p}{2-p}+1
&\geq
%\left(1-\frac{3}{4}\right)\frac {4p}{2-p}+1
%=\frac {p}{2-p}+1=
\frac {2}{2-p} >0.  %>-1.
\end{align*}
\begin{comment}
and we set
\begin{align*}%\label{1}
C=\left(
\int_{-1}^{1}
(1-\xi^2)^{\left(1-\frac{3}{s'}\right)\frac {4p}{2-p}+1}
d\xi
\right)^{\frac1p-\frac12}.
\end{align*}
\end{comment}
Therefore, we have
\begin{align*}%\label{1}
&\int_{-1}^{1}
(1-\xi^2)^{2-\frac6{s'}}
\Phi(\xi,k) d\xi
\\
&\leq
C
\left({2k+1}\right)^{\gamma\left(\frac1p\right)-\left(\frac1p-\frac12\right)}
\left(
\int_{-\infty}^\infty
\left(
\int_{\mathfrak v_\omega}
\left|\mathfrak F_\omega 
(\beta\circ R_\omega^{-1})
\big(x_\omega,y_\omega; \xi\big)\right|^p dx_\omega
dy_\omega
\right)^{\frac{p'}p}
d\xi
\right)^{\frac1{p'}}.
\end{align*}
Now, being $p'/p \geq 1$,
we
may use the Minkowski integral inequality,
which gives
%\footnote{\color{blue}La seconda riga \`e da saltare.}
\begin{align*}%\label{1}
&\int_{-1}^{1}
(1-\xi^2)^{2-\frac6{s'}}
\Phi(\xi,k) d\xi
%\\&\leq C \left({2k+1}\right)^{\gamma\left(\frac1p\right)-\left(\frac1p-\frac12\right)}\left(\int_{-\infty}^\infty\left(\int_{\mathfrak v_\omega}\left|\mathfrak F_\omega (\beta\circ R_\omega^{-1})\big(x_\omega,y_\omega; \xi\big)\right|^p dx_\omegady_\omega\right)^{\frac{p'}p}d\xi\right)^{\frac1{p'}}
\\
&\leq
C
\left({2k+1}\right)^{\gamma\left(\frac1p\right)-\left(\frac1p-\frac12\right)}
\left(
\int_{\mathfrak v_\omega}
\left(
\int_{-\infty}^\infty
\left|\mathfrak F_\omega 
(\beta\circ R_\omega^{-1})
\big(x_\omega,y_\omega; \xi\big)\right|^{p'}
d\xi
\right)^{\frac p{p'}}
dx_\omega
dy_\omega
\right)^{\frac1{p}}.
\end{align*}
Then we apply the Hausdorff-Young inequality
in the inner integral
to $\mathfrak F_\omega$ and replace
the coordinates $x_\omega,y_\omega, v_\omega$
with $x,y,v$,
%\footnote{\color{blue}Cancellare la seconda riga.}
\begin{align*}%\label{4}
&\int_{-1}^{1}
(1-\xi^2)^{2-\frac6{s'}}
\Phi(\xi,k) d\xi
%\\\notag&\leq C\left({2k+1}\right)^{\gamma\left(\frac1p\right)-\left(\frac1p-\frac12\right)}\left(\int_{\mathfrak v_\omega}\left(\int_{-\infty}^\infty\left|\mathfrak F_\omega (\beta\circ R_\omega^{-1})\big(x_\omega,y_\omega; \xi\big)\right|^{p'}d\xi\right)^{\frac p{p'}}dx_\omegady_\omega\right)^{\frac1{p}}
\\\notag
&\leq
C
\left({2k+1}\right)^{\gamma\left(\frac1p\right)-\left(\frac1p-\frac12\right)}
\left(
\int_{\mathfrak v_\omega}
\left(
\int_{-\infty}^\infty
\left| 
(\beta\circ R_\omega^{-1})
\big(x_\omega,y_\omega, v_\omega)\right|^{p}
dv_\omega
\right)
dx_\omega
dy_\omega
\right)^{\frac1{p}}
\\\notag
&\leq
C
\left({2k+1}\right)^{\gamma\left(\frac1p\right)-\left(\frac1p-\frac12\right)}
\left(
\int_{\mathfrak v}
\left| 
\beta
\big(x,y, v)\right|^{p}
dx
dydv
\right)^{\frac1{p}},
\end{align*}
deducing that
\begin{align}\label{4}
%&
\int_{-1}^{1}
(1-\xi^2)^{2-\frac6{s'}}
\Phi(\xi,k) d\xi
%\\\notag&\qquad\qquad\qquad
\leq
C
\left({2k+1}\right)^{\gamma\left(\frac1p\right)-\left(\frac1p-\frac12\right)}
\|\beta\|_{L^p_\mathfrak v}.
\end{align}
Finally, plugging \eqref{4} in \eqref{5} we obtain
\begin{align}\label{6}
|\langle \mathcal P_1 f, g \rangle|
&\leq C
\|\alpha\|_{L^s_\mathfrak z} \|\beta\|_{L^p_\mathfrak v}
\sum\limits_{k=0}^\infty
(2k+1)^{\frac6{s'} - 3+\gamma\left(\frac1p\right)-\left(\frac1p-\frac12\right)}
\\\notag
&\leq C
\|\alpha\|_{L^s_\mathfrak z} \|\beta\|_{L^p_\mathfrak v}
%\\\notag
= C
\|f\|_{L^s_\mathfrak zL^p_\mathfrak v}.
\end{align}
The series in fact converges for $1\leq s\leq 4/3$ and $1\leq p\leq2$ since by \eqref{gamma}, when $1\leq p\leq \frac65$,
we have
%\footnote{\color{blue}Questo paragrafo \`e da cancellare.}
\begin{align*}
\frac6{s'} - 3+\gamma\left(\frac1p\right)-\left(\frac1p-\frac12\right)
%&\leq-\frac32 +\frac1p-1-\left(\frac1p-\frac12\right)
\leq
-2
\end{align*}
and, when $\frac65\leq p\leq2$,
\begin{align*}
\frac6{s'} - 3+\gamma\left(\frac1p\right)-\left(\frac1p-\frac12\right)
%&\leq-\frac32 -\frac12\left(\frac1p-\frac12\right)-\left(\frac1p-\frac12\right)
%\\&=-\frac34-\frac32\frac1p
&\leq
-\frac32.
\end{align*}
From \eqref{6} the estimate asserted in the statement
follows by duality, proving the theorem.
\end{proof}

\end{document}